\documentclass[sn-mathphys-num]{sn-jnl}

\usepackage{mathalpha}
\usepackage{enumerate}
\usepackage{graphicx}%
\usepackage{multirow}%
\usepackage{amsmath,amssymb,amsfonts}%
\usepackage{amsthm}%
\usepackage{mathrsfs}%
\usepackage[title]{appendix}%
\usepackage{xcolor}%
\usepackage{textcomp}%
\usepackage{manyfoot}%
\usepackage{booktabs}%
\usepackage{algorithm}%
\usepackage{algorithmicx}%
\usepackage{algpseudocode}%
\usepackage{listings}%
\def\vrho{\varrho}

\def\f{\longrightarrow}
\def\im{\Longrightarrow}
\def\cim{\Longleftarrow}

\def\N{\mathbb{N}}

\def\e{\varepsilon}

\def\<{\langle}
\def\>{\rangle}

\def\g{\gamma}
  
\def\e{\varepsilon}

\def\R{\mathbb{R}}

\def\sx{{s_x}}

\def\O{\mathcal{O}}
\def\inte{\textnormal{int}\,}
\def\clo{\textnormal{cl}\,}

\def\bdry{\textnormal{bdry}\,}

\def\proj{\textnormal{proj}\sp}

\def\nvarphi{\bp\nvarphi}
\def\sp{\hspace{0.015cm}}
\def\bp{\hspace{-0.08cm}}

\theoremstyle{plain}%
\newtheorem{theorem}{Theorem}
\newtheorem{proposition}{Proposition}%
\newtheorem{corollary}{Corollary}

\newtheoremstyle{iremark}
  {\topsep}   
  {\topsep}   
  {\upshape}  
  {0pt}       
  {\itshape}  
  {.}         
  {5pt plus 1pt minus 1pt} 
  {\thmname{#1}\thmnumber{ \itshape#2}\thmnote{ (#3)}} 

\theoremstyle{iremark}
\newtheorem{remark}{Remark}

\theoremstyle{definition}%
\newtheorem{example}{Example}%

\begin{document}

\title[The Complement of a Closed Set Satisfying the Extended Exterior Sphere Condition]{The Complement of a Closed Set Satisfying the Extended Exterior Sphere Condition}


\author*[1]{\fnm{Chadi} \sur{Nour}}\email{cnour@lau.edu.lb}
\equalcont{These authors contributed equally to this work.}

\author[1]{\fnm{Jean} \sur{Takche}}\email{jtakchi@lau.edu.lb}
\equalcont{These authors contributed equally to this work.}


\affil[1]{\orgdiv{Department of Computer Science and Mathematics}, \orgname{Lebanese American University}, \orgaddress{\city{Byblos Campus}, \postcode{P.O. Box 36}, \state{Byblos}, \country{Lebanon}}}

%


\abstract{We provide a novel {\it analytical} proof of an {\it improved} version of \cite[Theorem 3.1]{NT2025}, showing that the complement of a closed set satisfying the extended exterior sphere condition is nothing but the union of closed balls with lower semicontinuous radius function. The improvement lies in the radius function, which is now {\it larger} than the one used in \cite[Theorem 3.1]{NT2025}.}




\maketitle

\section{Introduction}\label{sec1} Let $S\subset\R^n$ be a nonempty and closed set. For $s\in S$, we denote by $N_S^P(s)$ the {\it proximal normal cone} to $S$ at $s$, that is, the set of all vectors $\zeta\in\R^n$ satisfying, for certain $\sigma=\sigma(s,\zeta)\geq 0$, the inequality  \begin{equation}\label{proxineq} \<\zeta,x-s\>\leq \sigma\|x-s\|^2,\;\;\forall x\in S.\end{equation}
This latter inequality \eqref{proxineq} is commonly referred to as the {\it proximal normal inequality}, and is equivalent, for $\zeta\not=0$ and $\sigma>0$, to \begin{equation} \label{balls} B\left(s+\frac{1}{2\sigma}\frac{\zeta}{\|\zeta\|};\frac{1}{2\sigma}\right)\cap S=\emptyset,\end{equation}
where $B(x;\delta)$ denotes the open ball of radius $\delta$ centered at $x$. In that case, the normal vector $\zeta$ is said to be {\it realized by a $\frac{1}{2\sigma}$-sphere}.

For $r>0$, the set $S$ is said to be {\it $r$-prox-regular} if for all $s\in\bdry S$, the boundary of $S$, and $0\not=\zeta\in N_S^P(s)$, $\zeta$ is realized by an $r$-sphere, that is,  $$B\left(s+r\frac{\zeta}{\|\zeta\|};r\right)\cap S=\emptyset\;\;\left[\hbox{or equivalently}\; \left\<\frac{\zeta}{\|\zeta\|},x-s\right\>\leq \frac{1}{2r}\|x-s\|^2,\;\;\forall x\in S\right].$$
This geometric property, which is known to enjoy in a {\it neighborhood} some properties that {\it convex sets} satisfy {\it globally}, is also  referred to as {\it positive reach}, {\it proximal smoothness}, {\it $p$-convexity} and {\it $\varphi_0$-convexity}, see \cite{canino,csw,cm,fed,prt,shapiro}. 

 In \cite[Theorem 3]{Nacry}, Nacry and Thibault proved, in the context of general Hilbert space, that if $S$ is $r$-prox-regular, then for any $\e\in]0,1[$, the complement of $S$, denoted by $S^c$, is  the union of closed balls with the {\it common} radius $\e r$. Such result has been recently extended by Nour and Takche in \cite{NT2024,NT2025} in several directions. Indeed, in \cite[Theorem 1.2]{NT2024}, Nour and Takche proved that if $S$ satisfies the {\it extended exterior $r$-sphere condition}, then $S^c$ is nothing but the union of closed balls with common radius $\frac{r}{2}$.  Note that the $r$-prox-regularity is {\it stronger} than  the extended exterior $r$-sphere condition because,  in the latter's definition, {\it only one} proximal normal vector $\zeta\in N_S^P(s)$ is needed to be realized by an $r$-sphere for the boundary points $s\in\clo (\inte S)$, the closure of the interior of $S$.  More details about the comparison between these two properties can be found in \cite{NT2009,NT2024}. In \cite[Theorem 3.1]{NT2025}, the same authors generalized \cite[Theorem 1.2]{NT2024} by allowing the radius $r$ in the definition of the extended exterior $r$-sphere condition to be {\it any} continuous function $r(\cdot)\colon\bdry S\f]0,+\infty]$. As conclusion, they obtained that $S^c$ is nothing but the union of closed balls with the radius function $\rho\colon S^c\f]0,+\infty]$ defined by \begin{equation} \label{radiusrho} \rho(x):=\min\left\{\frac{r(s)}{2} : s\in\proj_{S}(x)\right\},\end{equation} where $\proj_{S}(x)$ denotes the projection of $x$ on $S$. This means that for every $x\in S^c$, there exists $y_x\in S^c$ such that:
$$\begin{cases}x\in\bar{B}(y_x;\rho(x))\subset S^c,&\hbox{if}\; \rho(x)<+\infty,\\ x\in\bar{B}(x+\delta(y_x-x);\delta)\subset S^c\;\,\hbox{for all}\;\delta>0,&\hbox{if}\; \rho(x)=+\infty,\end{cases}$$ where $\bar{B}(z;\delta)$ is the closed ball centered at $z$ with radius $\delta$. Note that the {\it geometric} proof of this generalization given in \cite{NT2025} relied on the {\it balls} characterizations of  the extended exterior sphere condition and the union of closed balls property, making the proof complicated and lengthy, see \cite[Section 3]{NT2025}.

The goal of this paper is twofold: First, we improve the statement of \cite[Theorem 3.1]{NT2025} by introducing a radius function {\it larger} than the function $\rho(\cdot)$ defined in \eqref{radiusrho} (see Theorem \ref{mainth}). Second, in order to prove the latter, we employ {\it analytical} characterizations of the extended exterior sphere condition and the union of closed balls property, which simplify and shorten the proof.

In the next section, we present our basic notations and definitions, provide analytical characterizations of the extended exterior sphere condition and the union of closed balls property, and we state the main result of this paper, namely, Theorem \ref{mainth}. In this latter, we prove that the complement of a set satisfying the extended exterior $r(\cdot)$-sphere condition is nothing but the union of closed balls with lower semicontinuous radius function larger than the one used in \cite[Theorem 3.1]{NT2025}. Section \ref{sec3} is devoted to the analytical proof of Theorem \ref{mainth}.

\section{Preliminaries $-$ Main Result} \label{sec2} 

\subsection{Basic Notations and Definitions}

For the Euclidean norm and the usual inner product, we use $\|\cdot\|$ and $\<,\>$, respectively. For $r>0$ and $x\in\R^n$, we set $B(x;\rho):= x + \rho B$ and $\bar{B}(x;\rho):= x + \rho \bar{B}$, where $B$ and $\bar{B}$ are the open and the closed unit balls, respectively.  For a set $S\subset\R^n$, we denote by $S^c$, $\inte S$, $\bdry S$ and $\clo S$, the complement (with respect to $\R^n$), the interior, the boundary and the closure of $S$, respectively. The closed segment (resp. open segment)  joining two points $x$ and $y$ in $\R^n$ is denoted by $[x,y]$ (resp. $]x,y[$). The distance from a point $x$ to a set $S$ is denoted by $d_S(x)$. We also denote by $\proj_S(x)$ the set of closest points in $S$ to $x$, that is, the set of points $s$ in $S$ satisfying $d_S(x)=\|s-x\|$. Finally and as mentioned in the introduction, we denote, for nonempty and closed set $S\subset\R^n$ and $s\in S$, by $N_S^P(s)$ the proximal normal cone to $S$ at $s$.  Note that for $x\in S^c$, $s_x\in\proj_S(x)$ and $\zeta_{s_x}:=\frac{x-s_x}{\|x-s_x\|}$, we have that $\zeta_{s_x}\in N_S^P(s_x)$, and $\zeta_{s_x}$ is realized by a $\|x-s_x\|$-sphere. More information about proximal and nonsmooth analysis can be found in the monographs \cite{clsw,mord,penot,rockwet,thibault}.

\subsection{The Extended Exterior Sphere Condition}

Let $S\subset\R^n$ be nonempty and closed. For $r(\cdot)\colon\bdry S\f]0,+\infty]$ continuous, we say that $S$ satisfies the  {\it extended exterior $r(\cdot)$-sphere condition} if for every $s\in\bdry S$, the following assertions hold:
\begin{itemize}
\item If $s\in\bdry(\inte S)$ then there {\it exists} a unit vector $\zeta_s\in N^P_S(s)$ such that $$\begin{cases}\zeta_s\;\hbox{is realized by an}\;r(s)\hbox{-sphere},&\hbox{if}\; r(s)<+\infty,\\ \zeta_s\;\hbox{is realized by a}\;\rho\hbox{-sphere}\;\hbox{for all}\;\rho>0,&\hbox{if}\; r(s)=+\infty.\end{cases}$$
\item  If $s\not \in\bdry(\inte S)$ then for {\it all} unit vectors $\zeta_s\in N^P_S(s)$, we have $$\begin{cases}\zeta_s\;\hbox{is realized by an}\;r(s)\hbox{-sphere},&\hbox{if}\; r(s)<+\infty,\\ \zeta_s\;\hbox{is realized by a}\;\rho\hbox{-sphere}\;\hbox{for all}\;\rho>0,&\hbox{if}\; r(s)=+\infty.\end{cases}$$
\end{itemize}
From the equivalences  (for $\zeta\not=0$ and $\sigma>0$) \begin{equation} \label{equivalence} \eqref{proxineq} \Longleftrightarrow \eqref{balls} \Longleftrightarrow \underbrace{\left\< x-s,x-s-\frac{1}{\sigma}\frac{\zeta}{\|\zeta\|}\right\> \geq 0,\;\;\forall x\in S,}_{(*)}\end{equation}
we deduce the following analytical characterization of the extended exterior sphere condition. The set $S$ satisfies the extended exterior $r(\cdot)$-sphere condition if for every $s\in\bdry S$, the following assertions hold:
\begin{itemize}
\item If $s\in\bdry(\inte S)$ then there exists a unit vector $\zeta_s\in N^P_S(s)$ such that $$\begin{cases}\left\< \zeta_s,x-s\right\> \leq \frac{1}{2r(s)} \|x-s\|^2\;\,\hbox{for all}\;x\in S,&\hbox{if}\; r(s)<+\infty,\\  \left\< \zeta_s,x-s\right\>\leq 0\;\,\hbox{for all}\;x\in S,&\hbox{if}\; r(s)=+\infty.\end{cases}$$ $$\left[\hbox{or equivalently} \begin{cases}\left\<x-s,x-s-2r(s)\zeta_s\right\> \geq 0\;\,\hbox{for all}\;x\in S,&\hbox{if}\; r(s)<+\infty,\\  \left\< \zeta_s,x-s\right\>\leq 0\;\,\hbox{for all}\;x\in S,&\hbox{if}\; r(s)=+\infty.\end{cases}\right]$$
\item  If $s\not \in\bdry(\inte S)$ then for all unit vectors $\zeta_s\in N^P_S(s)$, we have $$\begin{cases}\left\< \zeta_s,x-s\right\> \leq \frac{1}{2r(s)} \|x-s\|^2\;\,\hbox{for all}\;x\in S,&\hbox{if}\; r(s)<+\infty,\\  \left\< \zeta_s,x-s\right\> \leq 0\;\,\hbox{for all}\;x\in S,&\hbox{if}\; r(s)=+\infty. \end{cases}$$ $$\left[\hbox{or equivalently} \begin{cases}\left\<x-s,x-s-2r(s)\zeta_s\right\> \geq 0\;\,\hbox{for all}\;x\in S,&\hbox{if}\; r(s)<+\infty,\\  \left\< \zeta_s,x-s\right\>\leq 0\;\,\hbox{for all}\;x\in S,&\hbox{if}\; r(s)=+\infty.\end{cases}\right]$$
\end{itemize}
As mentioned in the introduction, the extended exterior $r(\cdot)$-sphere condition is {\it weaker} than the $r(\cdot)$-prox-regularity. More information about this property can be found in \cite{NT2024,NT2025}.\\

\begin{remark} \label{remeq} The equivalences of \eqref{equivalence} remain valid if the inequalities in \eqref{proxineq} and $(*)$ are strict, and the open ball in \eqref{balls} is replaced by a closed one.\\
\end{remark}

\begin{remark} \label{remeqbis} Using the equivalences of \eqref{equivalence} and Remark \ref{remeq}, one can easily prove that for  $S\subset\R^n$ nonempty and closed, and for $x\in\R^n$ and $\delta>0$, if $y$ and $z$ are two diametrically opposite points of the closed ball $\bar{B}(x;\delta)$, then \begin{equation*} \label{equivrem2} \bar{B}(x;\delta)\subset S^c\;\;(\hbox{resp.}\;{B}(x;\delta)\subset S^c)  \Longleftrightarrow \<s-y,s-z\>>0\;\;(\hbox{resp.}\;\geq 0),\;\;\forall s\in S.\end{equation*}
This can also be easily deduced from the following fact $$\forall s\in\R^n,\;\hbox{we have}\;\begin{cases}\<s-y,s-z\><0, &\hbox{if}\;s\in B(x;\delta),\\\<s-y,s-z\>=0, &\hbox{if}\;s\in \bdry (B(x;\delta)),\\ \<s-y,s-z\>>0, &\hbox{if}\;s\not \in \bar{B}(x;\delta). \end{cases}$$
From this latter, we can also deduce that for all $(y,z,y',z')\in\R^{4n}$, if $[y',z']\subset ]y,z[ $ then \begin{equation} \label{equivrembis} \forall s\in\R^n,\;\big[ \<s-y,s-z\>\geq0\im \<s-y',s-z'\>>0\big]. \end{equation}

\end{remark}

\subsection{The Union of Closed Balls Property}  Let $\O\subset\R^n$ be nonempty and open. For $\varrho\colon\O\f]0,+\infty]$ a lower semicontinuous function, we say that $\O$ is the union of closed balls with radius function $\varrho(\cdot)$ if for every $x\in\O$, there exists $y_x\in\O$ such that 
\begin{equation}\label{ucbp} \begin{cases}x\in\bar{B}(y_x;\varrho(x))\subset \O,&\hbox{if}\; \varrho(x)<+\infty,\\ x\in\bar{B}(x+\delta(y_x-x);\delta)\subset \O\;\,\hbox{for all}\;\delta>0,&\hbox{if}\; \varrho(x)=+\infty.\end{cases}\end{equation}
The following proposition provides an analytical characterization of the union of closed balls property introduced above. This characterization will be essential for proving our main result.\\

\begin{proposition} \label{prop1} Let $\O\subset\R^n$ be nonempty and open.  For $\varrho\colon\O\f(0,+\infty]$ a lower semicontinuous function, the set $\O$ is the union of closed balls with radius function $\varrho(\cdot)$  if and only if for every $x\in\O$ there exists a unit vector $\zeta_x\in\R^n$ satisfying$\sp:$
\begin{enumerate}[$(i)$]
\item If $\varrho(x)<+\infty$, then there exists  $t_x\in [0,\varrho(x)]$ such that  \begin{equation} \label{imtri} \<y-x+t_x\zeta_x,y-x+(t_x-2\vrho(x))\zeta_x\> >0,\;\;\forall y\in\O^c.\end{equation}
\item If $\varrho(x)=+\infty$, then \begin{equation} \label{imtribis} \<\zeta_x,y-x\>\leq0,\;\;\forall y\in\O^c. \end{equation}
\end{enumerate}
\end{proposition}
\begin{proof} $\cim:$ Let $x\in\O$, and let $\zeta_x\in\R^n$ unit such that Proposition \ref{prop1}$(i)$-$(ii)$ are satisfied. There are two cases to consider.\vspace{0.1cm}\\
\underline{Case 1:} $\varrho(x)<+\infty$.\vspace{0.1cm}\\
Let $y_x:=x+(\vrho(x)-t_x)\zeta_x$. Then, using \eqref{imtri}, we have $$\<y-y_x+\vrho(x)\zeta_x,y-y_x-\vrho(x)\zeta_x\>>0,\;\;\forall y\in\O^c.$$
This yields, using Remark \ref{remeq}, that $\bar{B}(y_x;\varrho(x))\subset (\O^c)^c=\O$.\vspace{0.1cm}\\
\underline{Case 2:} $\varrho(x)=+\infty$.\vspace{0.1cm}\\
Then  $\<\zeta_x,y-x\>\leq 0< \frac{1}{2\delta}\|y-x\|^2$ for all $\delta>0$ and $y\in\O^c$. Hence, using  Remark \ref{remeq}, we deduce that for $y_x:=x+\zeta_x$, we have $$B(x+\delta(y_x-x);\delta)=B(x+\delta\zeta_x;\delta)\subset (\O^c)^c=\O,\;\;\forall \delta>0.$$
$\im:$  Let $x\in\O$, and let  $y_x\in\O$ such that \eqref{ucbp} is satisfied. There are two cases to consider.\vspace{0.1cm}\\
\underline{Case 1:} $\varrho(x)<+\infty$.\vspace{0.1cm}\\
Then $x\in\bar{B}(y_x;\varrho(x))\subset \O$. This yields that $t_x:=(\vrho(x)-\|y_x-x\|)\in[0,\vrho(x)]$. Moreover, for $$\zeta_x:=\begin{cases}\displaystyle \frac{y_x-x}{\|y_x-x\|}, &\hbox{if}\;x\not=y_x,\vspace{0.15cm}\\ (1,0,\dots,0)\in\R^n, &\hbox{if}\;x=y_x,\end{cases} $$
we have \begin{eqnarray*}B((x-t_x\zeta_x)+\vrho(x)\zeta_x;\vrho(x))\cap \O^c=B(y_x;\vrho(x))\cap\O^c\subset \bar{B}(y_x;\vrho(x))\cap\O^c=\emptyset. \end{eqnarray*}
This gives, using Remark \ref{remeq}, that $$ \<y-x+t_x\zeta_x,y-x+(t_x-2\vrho(x))\zeta_x\> >0,\;\;\forall y\in\O^c. $$
\underline{Case 2:} $\varrho(x)=+\infty$.\vspace{0.1cm}\\
Then  $x\in\bar{B}(x+\delta(y_x-x);\delta)\subset \O$ for all $\delta>0$. Clearly, we have that $\|y_x-x\|\leq1$. 
\underline{Case 2.1:} $y_x=x$.\vspace{0.1cm}\\
Then $\bar{B}(x;\delta)\subset \O$  for all  $\delta>0$, which yields that $\O=\R^n$. Now taking $\zeta_x$ any unit vector, we clearly have, as $\O^c=\emptyset$, that $$\<\zeta_x,y-x\>\leq0,\;\;\forall y\in\O^c.$$
\underline{Case 2.2:} $y_x\not=x$.\vspace{0.1cm}\\
Let $\zeta_x:= \frac{y_x-x}{\|y_x-x\|}$. For all $\delta>0$, we have \begin{eqnarray*}\bar{B}(x+\delta\|y_x-x\|\zeta_x; \|y_x-x\|\delta)\cap 
\O^c&=& \bar{B}(x+\delta(y_x-x); \|y_x-x\|\delta)\cap\O^c \\&\subset & \bar{B}(x+\delta(y_x-x); \delta)\cap\O^c=\emptyset. \end{eqnarray*}
Hence, by Remark \ref{remeq}, we deduce that for all $\delta>0$, we have  $$\<\zeta_x,x-y\><\frac{1}{2\delta\|y_x-x\|\\}\|y-x\|^2,\;\;\forall y\in\O^c.$$
Taking $\delta\f+\infty$, we conclude that $\<\zeta_x,y-x\>\leq0$ for all $y\in\O^c$.\end{proof}

\begin{example} In this example, we prove that the inequality \eqref{imtribis} cannot be strict. In $\R^2$, let $\O:=S^c$ where $S$ is the closed set 
$$S:=\{(x,y)\in\R^2 : x\in]-\infty,-1]\cup [1,+\infty[\;\hbox{and}\;y=0\}.$$ Clearly, $\O$ is the union of closed balls with radius function $\vrho(x,y):=+\infty$ for all $(x,y)\in\O$. Consider $(0,0)\in\O$ and $\zeta=(\zeta_1,\zeta_2)$ unit such that $$\<(\zeta_1,\zeta_2),(x,y)-(0,0)\>\leq 0,\;\;\forall (x,y)\in \O^c.$$
Then $\zeta_1x+\zeta_2y\leq0$ for all $(x,y)\in \O^c$. Taking $(x,y)=(\pm1,0)\in\O^c$, we obtain that $\zeta_1= 0$. Hence, $\zeta=(0,\pm1)$. Now, since $(1,0)\in\O^c$, we have that $$\<\zeta,(1,0)-(0,0)\>= \<(0,\pm1),(1,0)\>=0.$$
\end{example} 

\subsection{Statement of the Main Result} We begin by introducing some notations needed for the statement of our main result. Let $S\subset\R^n$ be nonempty and closed, and let $r(\cdot)\colon\bdry S\f]0,+\infty]$ be continuous. We denote by $\rho\colon S^c\f]0,+\infty]$ the function defined in \eqref{radiusrho}, that is, $$\rho(x):=\min\left\{\frac{r(s)}{2} : s\in\proj_{S}(x)\right\}.$$
Note that the function $\rho$ is lower semicontinuous on $S^c$, and the infimum defining $\rho(x)$, for each $x\in S^c$, is attained, see \cite[Proof of Theorem 3.1]{NT2025}. So, for each $x\in S^c$, we denote by $s_x\in\proj_S(x)$ a point such that $ \rho(x)=\frac{r(s_x)}{2}$.
For $\g>0$, and $\rho_x:=d_S(x)$ for all $x\in S^c$, we define the function $\vrho_\g\colon S^c\f]0,+\infty]$ by \begin{equation}\label{radiusvrhoe} \vrho_\g(x):=\max\left\{\g\rho_x,\frac{1}{2}\sqrt{\g^2\rho_x^2+4\rho(x)^2}\right\}. \end{equation}

Now we are ready to state the main result of this paper, namely Theorem \ref{mainth}. Note that this theorem, whose proof is deferred  to the next section, generalizes  \cite[Theorem 3.1]{NT2025} as we will illustrate in Corollary \ref{coro}.\\

\begin{theorem} \label{mainth} Let $S\subset\R^n$ be a nonempty and closed set satisfying the extended exterior $r(\cdot)$-sphere condition for some continuous function $r(\cdot)\colon\bdry S\f]0,+\infty]$. Then for all $\frac{1}{2\sqrt{3}-2}\leq \g<1$, $S^c$ is the union of closed balls with lower semicontinuous radius function $\vrho_\g(\cdot)$.\\
\end{theorem}

From the definition of the function $\vrho_\g(\cdot)$ given in \eqref{radiusvrhoe}, we deduce that $\vrho_\g(x)\geq \rho(x)$ for all $x\in S^c$. Then Theorem  \ref{mainth}  gives rise to the following corollary, which coincides with \cite[Theorem 3.1]{NT2025}.\\

\begin{corollary}[{\cite[Theorem 3.1]{NT2025}}] \label{coro}  Let $S\subset\R^n$ be a nonempty and closed set satisfying the extended exterior $r(\cdot)$-sphere condition for some continuous function $r(\cdot)\colon\bdry S\f]0,+\infty]$. Then $S^c$ is the union of closed balls with lower semicontinuous radius function $\rho(\cdot)$.\\
\end{corollary}

\begin{example} \label{remexam} In this example, we prove that the constant $\g$ in Theorem  \ref{mainth} {\it cannot} be replaced by $1$. We consider in $\R^2$ the three points $$c_1:=\left(-\tfrac{2}{\sqrt{3}},0\right),\;c_2:=\left(\tfrac{1}{\sqrt{3}},-1\right),\;\hbox{and}\;c_3:= \left(\tfrac{1}{\sqrt{3}},1\right).$$
We define $S:=(B(c_1;1))^c\cap(B(c_1;1))^c\cap(B(c_3;1))^c$, see Fig. \ref{Fig1}. 
One can easily see that $S$ satisfies the extended exterior $1$-sphere condition. Furthermore:
\begin{itemize}
\item The {\it largest} radius of an {\it open} ball in $S^c=B(c_1;1)\cup B(c_1;1)\cup B(c_3;1)$ containing $c_1$ is $\rho_{c_1}=1$. 
\item For the origin $(0,0)$, we have $\rho_{(0,0)}=\frac{1}{\sqrt{3}}$ and $\rho(0,0)=\frac{1}{2}$. Moreover, the {\it largest} radius of an {\it open} ball in $S^c$ containing $(0,0)$ is $$\frac{1}{\sqrt{3}}=\frac{1}{2}\sqrt{\rho_{(0,0)}^2+4\rho(0,0)^2}.$$
\end{itemize}
\end{example}
\begin{figure}[tb]
\centering
\includegraphics[scale=1.41]{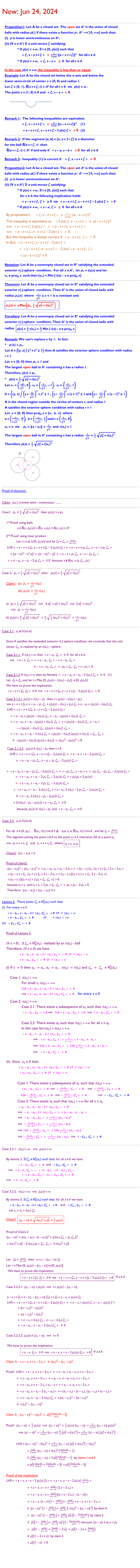}
\caption{\label{Fig1} The set $S$ of Example \ref{remexam}}
\end{figure}
\section{Proof of the Main Result} \label{sec3} The goal of this section is to provide an analytical proof for Theorem \ref{mainth}. Let $S\subset\R^n$ be a nonempty and closed set satisfying the extended exterior $r(\cdot)$-sphere condition for some continuous function $r(\cdot)\colon\bdry S\f]0,+\infty]$. We fix $\frac{1}{2\sqrt{3}-2}\leq \g<1$.\\
{\it Claim} 1: The function $\vrho_\gamma(\cdot)$ is lower semicontinuous.\vspace{0.1cm}\\
As the maximum of two lower semicontinuous functions is known to be  lower semicontinuous, and the function $x\mapsto \rho_x$ is continuous,  it is sufficient to prove that the function $x\mapsto \sqrt{\g^2\rho_x^2+4\rho(x)^2}$ is lower semicontinuous. This latter follows directly from the continuity of $x\mapsto \rho_x$, the lower semicontinuity of $\rho(\cdot)$, and since the sum of two lower semicontinuous functions is lower semicontinuous, and the two functions  $x\mapsto x^2$ and $x\mapsto \sqrt{x}$ are both continuous and monotonically nondecreasing on $[0,+\infty]$. This terminates the proof of Claim 1.

We proceed to prove that $S^c$ is the union of closed balls with radius function $\vrho_\g(\cdot)$. Let $x\in S^c$, and denote by  $s_x\in\proj_S(x)$ the point satisfying $\rho(x)=\frac{r(s_x)}{2}$, and by $\zeta_{s_x}:=\frac{x-s_x}{\|x-s_x\|}\in N_S^P(s_x)$. There are two cases to consider.\vspace{0.1cm}\\
\underline{Case 1:} $\g \rho_x\geq \frac{1}{2}\sqrt{\g^2\rho_x^2+4\rho(x)^2}$.\vspace{0.1cm}\\
Then $\vrho_\g(x)=\g\rho_x<+\infty$. Let $t_x:=\g\rho_x\in[0,\vrho_\g(x)]$. Since $\zeta_{s_x}$ is realized by a $\rho_x$-sphere, we have from \eqref{equivalence} that  $$\<s-s_x,s-s_x-2\rho_x\zeta_{s_x}\>\geq 0,\;\;\forall s\in S.$$
This yields, using \eqref{equivrembis} and the inclusion \begin{eqnarray*} [x-t_x\zeta_{s_x},x-t_x\zeta_{s_x}+2\vrho_\g(x)\zeta_{s_x}] &=&[x-\g\rho_x\zeta_{s_x},x+\g\rho_x\zeta_{s_x}]\\&\subset& ]x-\rho_x\zeta_{s_x},x+\rho_x\zeta_{s_x}[\\&=&]s_x,s_x+2\rho_x\zeta_{s_x}[,\end{eqnarray*} 
that $$\<s-x+t_x\zeta_{s_x},s-x+(t_x-2\vrho_\g(x))\zeta_{s_x}\>> 0,\;\;\forall s\in S.$$
Hence, for $\zeta_x:=\zeta_{s_x}$, we have $$\<s-x+t_x\zeta_x,s-x+(t_x-2\vrho_\g(x))\zeta_x\>> 0,\;\;\forall s\in S.$$
\underline{Case 2:} $\g \rho_x< \frac{1}{2}\sqrt{\g^2\rho_x^2+4\rho(x)^2}$.\vspace{0.1cm}\\
Then $\vrho_\g(x)=  \frac{1}{2}\sqrt{\g^2\rho_x^2+4\rho(x)^2}$. Moreover, it follows that  \begin{equation}\label{inequalities}   \begin{cases} \displaystyle  \rho_x<\frac{2}{\g \sqrt{3}}\rho(x)=\frac{1}{\g \sqrt{3}}r(s_x),\vspace{0.15cm}\\ \displaystyle  \vrho_\g(x)<\frac{2}{\sqrt{3}}\rho(x)=\frac{1}{\sqrt{3}}r(s_x),\,\hbox{and}\vspace{0.15cm}\\ \displaystyle  \rho_x+2\vrho_\g(x)<\left(\frac{1}{\g}+2\right)\frac{2}{\sqrt{3}}\rho(x)\leq4\rho(x),\end{cases} \end{equation}
where the last inequality follows since $\g\geq \frac{1}{2\sqrt{3}-2}$.\vspace{0.1cm}\\
\underline{Case 2.1:} $s_x\not\in\bdry(\inte S)$.\vspace{0.1cm}\\
Then the vector $\zeta_{s_x}\in N_S^P(s_x)$ is realized by an $r(s_x)$-sphere.\vspace{0.1cm}\\
\underline{Case 2.1.1:} $r(s_x)=2\rho(x)=+\infty$.\vspace{0.1cm}\\
Then $\vrho_\g(x)=+\infty$, and $ \<\zeta_{s_x},s-s_x\>\leq0$ for all $s\in S$. Hence, for $\zeta_x:=\zeta_{s_x}$, we have $$\<\zeta_x,s-x\>=\<\zeta_x,s-s_x\>-\<\zeta_x,x-s_x\>\leq -\<\zeta_x,x-s_x\>=-\rho_x<0,\;\;\forall s\in S.$$
\underline{Case 2.1.2:} $r(s_x)=2\rho(x)<+\infty$.\vspace{0.1cm}\\
Then $\vrho_\g(x)<+\infty$,  and $\left\<s-s_x,s-s_x-2r(s_x)\zeta_{s_x}\right\> \geq 0$ for all $s\in S$. Hence, for $\zeta_x:=\zeta_{s_x}$, we have 
$$\left\<s-s_x,s-s_x-4\rho(x)\zeta_{x}\right\> \geq 0,\;\;\forall s\in S.$$
This yields, using \eqref{equivrembis} and the inclusion $$[x,x+2\vrho_\g(x)\zeta_x]=[s_x+\rho_x\zeta_x,s_x+(\rho_x+2\vrho_\g(x))\zeta_x]\overset{\eqref{inequalities}}{\subset} ]s_x,s_x+4\rho(x)\zeta_x[,$$
that, for $t_x:=0\in [0,\varrho_\g(x)]$, $$\<s-x+t_x\zeta_x,s-x+(t_x-2\vrho_\g(x))\zeta_x\>=\<s-x,s-x-2\vrho_\g(x)\zeta_x\>> 0,\;\;\forall s\in S.$$
\underline{Case 2.2:} $s_x\in\bdry(\inte S)$.\vspace{0.1cm}\\
Let $N_x\in\N$ such that $0<\frac{1}{N_x}<\rho_x$. Having $s_x\in\bdry(\inte S)$, we deduce that $\bar{B}(s_x;\frac{1}{n})\cap \inte S\not=\emptyset$ for all $n\geq N_x$. For each $n\geq N_x$, we denote by $z_n\in \bar{B}(s_x;\frac{1}{n})\cap \inte S$ and by $\zeta_n:=\frac{z_n-x}{\|z_n-x\|}$. We also consider $s_n\in [x,z_n]\cap \bdry S$, where the latter intersection is nonempty since $x\in S^c$ and $z_n\in \inte S$. Clearly we have, for each $n\geq N_x$, $$s_n=x+t_1^n\zeta_n\;\;\hbox{and}\;\;z_n=x+t_2^n\zeta_n,\;\,\hbox{for some}\;\;t_2^n>t_1^n\geq\rho_x.$$
A simple calculation yields that for all $n\geq N_x$, $$\|z_n-s_x\|^2-\|s_n-s_x\|^2=(t_2^n-t_1^n)((t_2^n+t_1^n)+2\rho_x\<\xi_n,\zeta_x\>)>0,$$
where the last inequality follows since $t_2^n>t_1^n$, and $$t_2^n+t_1^n+2\rho_x\<\xi_n,\zeta_x\>>\rho_x+\rho_x-2\rho_x=0.$$
Hence, $\|s_n-s_x\|<\|z_n-s_x\|\leq\frac{1}{n}$ for all $n\geq N_x$. This gives that \begin{equation}\label{limits} \lim_{n\f+\infty} s_n=\lim_{n\f+\infty} z_n=s_x.\end{equation}
On the other hand, as $S$ satisfies the extended exterior $r(\cdot)$-sphere condition, there exists, for each $n\geq N_x$, a unit vector $\xi_n\in N_S^P(s_n)$ such that \begin{equation}\label{claim2(i)} \begin{cases}\left\<s-s_n,s-s_n-2r(s_n)\xi_n\right\> \geq 0\;\,\hbox{for all}\;s\in S,&\hbox{if}\; r(s_n)<+\infty,\\  \left\< \xi_n,s-s_n\right\>\leq 0\;\,\hbox{for all}\;s\in S,&\hbox{if}\; r(s_n)=+\infty.\end{cases}\end{equation}
From \eqref{limits}, the proximal normal inequality, the continuity of $r(\cdot)$, and since $\xi_n$ is unit, we can assume that $\xi_n\f\xi_{\sx}\in N_S^P(s_x)$. \vspace{0.1cm}\\
{\it Claim} 2: The unit vector $\xi_{s_x}$ satisfies the following:\vspace{0.1cm}
\begin{enumerate}[$(i)$]
\item $\begin{cases}\<s-s_x,s-s_x-2r(s_x)\xi_\sx\>\geq 0\;\,\hbox{for all}\,s\in S, &\hbox{if}\;r(s_x)<+\infty,\\\<\xi_\sx,s-s_x\>\leq 0\;\,\hbox{for all}\,s\in S, &\hbox{if}\;r(s_x)=+\infty. \end{cases}$\vspace{0.1cm}
\item $\<\xi_\sx,\zeta_\sx\>\geq 0$.\vspace{0.2cm}
\end{enumerate}
Note that Claim 2$(i)$ follows after taking $n\f\infty$ in \eqref{claim2(i)}. For Claim 2$(ii)$, it is sufficient to replace $s$ by $z_n$ in \eqref{claim2(i)}, then take $n\f\infty$ after noticing that $$\frac{z_n-s_n}{\|z_n-s_n\|}=\frac{z_n-x}{\|z_n-x\|}\f \zeta_{\sx}.$$
\underline{Case 2.2.1:} $r(s_x)=2\rho(x)=+\infty$.\vspace{0.1cm}\\
Then $\vrho_\g(x)=+\infty$, and by Claim 2, there exists a unit vector $\xi_{s_x}\in N_S^P(s_x)$ such that $$\<\xi_\sx,s-s_x\>\leq 0,\;\;\forall s\in S,\;\;\hbox{with}\;\;\<\xi_\sx,\zeta_\sx\>\geq 0.$$
Hence, for $\zeta_x:=\xi_\sx$, we have $$\<\zeta_x,s-x\> = \<\xi_\sx,s-s_x\>+ \<\xi_\sx,s_x-x\> \leq  \<\xi_\sx,s_x-x\>=-\rho_x\<\xi_\sx,\zeta_\sx\>\leq 0,\,\;\forall s\in S.$$
\underline{Case 2.2.2:} $r(s_x)=2\rho(x)<+\infty$.\vspace{0.1cm}\\
Then $\vrho_\g(x)<+\infty$, and by Claim 2, there exists a unit vector $\xi_{s_x}\in N_S^P(s_x)$ such that \begin{equation}\label{xiineq} \<s-s_x,s-s_x-2r(s_x)\xi_\sx\>\geq 0,\;\;\forall s\in S,\;\;\hbox{with}\;\;\<\xi_\sx,\zeta_\sx\>\geq 0.\end{equation}
Let $y_x:=s_x+r(s_x)\xi_\sx$. We define $$\zeta_x:=\begin{cases}\displaystyle  \frac{y_x-x}{\|y_x-x\|}, &\hbox{if}\;y_x\not=x, \vspace{0.15cm}\\ \xi_\sx=\zeta_\sx &\hbox{if}\;y_x=x, \end{cases}\;\;\hbox{and}\;\;t_x:=\max\{0,\vrho_\g(x)-\|y_x-x\|\}\in [0,\vrho_\g(x)].$$
\underline{Case 2.2.2.1:} $\|y_x-x\|\leq \vrho_\g(x)$.\vspace{0.1cm}\\
Then $t_x=(\vrho_\g(x)-\|y_x-x\|)\in [0,\vrho_\g(x)]$. Clearly we have $$x-t_x\zeta_x=y_x-\vrho_\g(x)\zeta_x\;\;\hbox{and}\;\;x+(2\vrho_\g(x)-t_x)\zeta_x=y_x+\vrho_\g(x)\zeta_x.$$
This yields that \begin{eqnarray}\nonumber [x-t_x\zeta_x,x+(2\vrho_\g(x)-t_x)\zeta_x]&=&[y_x-\vrho_\g(x)\zeta_x,y_x+\vrho_\g(x)\zeta_x]\\&\overset{\eqref{inequalities}}{\subset}& ]y_x-r(s_x)\zeta_x,y_x+r(s_x)\zeta_x[.\label{inclusionimp} \end{eqnarray}
From \eqref{xiineq} and the definition of $y_x$, we deduce that $$\<s-y_x+r(s_x)\xi_\sx,s-y_x-r(s_x)\xi_\sx\>\geq 0,\;\;\forall s\in S.$$
This gives using \eqref{inclusionimp} and \eqref{equivrembis}, that $$\<s-x+t_x\zeta_x, s-x+(t_x-2\vrho_\g(x))\zeta_x\>>0,\;\;\forall s\in S.$$
\underline{Case 2.2.2.2:} $\|y_x-x\|> \vrho_\g(x)$.\vspace{0.1cm}\\
Then $t_x=0\in [0,\vrho_\g(x)]$, and $\zeta_x=\frac{y_x-x}{\|y_x-x\|}$. So, we need to prove that $$\<s-x, s-x-2\vrho_\g(x)\zeta_x\>>0,\;\;\forall s\in S. $$
Let $s\in S$. Using the definition of $\zeta_x$, one can easily prove that \begin{eqnarray}\nonumber  \<s-x, s-x-2\vrho_\g(x)\zeta_x\>&=&\left(1-\frac{\vrho_\g(x)}{\|y_x-x\|}\right)\|s-x\|^2 \\ &+& \frac{\vrho_\g(x)}{\|y_x-x\|} \<s-x,s-x-2y_x\>.\label{last1}\end{eqnarray}
On the other hand, we have  \begin{eqnarray*} \<s-x,s-x-2y_x\> &=& \<s-s_x,s-s_x-2(y_x-s_x)\>\\&+&\<(s_x-y_x)-(x-y_x),(s_x-y_x)+(x-y_x)\>\\&=&\<s-s_x,s-s_x-2r(s_x)\xi_\sx\> \\&+& \|s_x-y_x\|^2-\|x-y_x\|^2 \\&\overset{\eqref{xiineq}}{\geq}& \|s_x-y_x\|^2-\|x-y_x\|^2 = r(s_x)^2-\|x-y_x\|^2.\end{eqnarray*}
Combining this latter with \eqref{last1}, we obtain that \begin{eqnarray}\nonumber \<s-x, s-x-2\vrho_\g(x)\zeta_x\>&\geq&  \left(1-\frac{\vrho_\g(x)}{\|y_x-x\|}\right)\|s-x\|^2\\&+& \frac{\vrho_\g(x)}{\|y_x-x\|}\left(r(s_x)^2-\|x-y_x\|^2\right).\label{last2}\end{eqnarray}
\underline{Case 2.2.2.2.1:} $\|y_x-x\|> 2\vrho_\g(x)$.\vspace{0.1cm}\\
Then $\left(1-\frac{\vrho_\g(x)}{\|y_x-x\|}\right)>0$. Moreover, we have  $\|s-x\|\geq d_S(x)=\rho_x$, and \begin{eqnarray*}r(s_x)^2-\|x-y_x\|^2&=& r(s_x)^2- \|(x-s_x)+(s_x-y_x)\|^2 \\ &=&r(s_x)^2-\|\rho_x\zeta_\sx-r(s_x)\xi_\sx\|^2\\&=& -\rho_x^2 +2r(s_x)\rho_x\<\xi_\sx,\zeta_\sx\>\overset{\eqref{xiineq}}{\geq} -\rho_x^2.\end{eqnarray*}
Hence, using \eqref{last2}, we conclude  that $$  \<s-x, s-x-2\vrho_\g(x)\zeta_x\> \geq \rho_x^2\left(1-\frac{2\vrho_\g(x)}{\|y_x-x\|}\right)> 0.$$
\underline{Case 2.2.2.2.2:} $\|y_x-x\|\leq 2\vrho_\g(x)$.\vspace{0.1cm}\\
Then \begin{eqnarray*}r(s_x)^2-\|x-y_x\|^2&\geq&  r(s_x)^2 - 2\vrho_\g(x)\|x-y_x\| \\&=& r(s_x)^2 - \frac{2}{\vrho_\g(x)}\|x-y_x\|\vrho_\g(x)^2\\&=&  r(s_x)^2 - \frac{2}{\vrho_\g(x)}\|x-y_x\| \frac{1}{4} \left(\g^2\rho_x^2+r(s_x)^2\right)\\&=& r(s_x)^2\left(1-\frac{\|x-y_x\|}{2\vrho_\g(x)}\right)-\g^2\rho_x^2\frac{\|x-y_x\|}{2\vrho_\g(x)}\\ &\overset{\eqref{inequalities}}{\geq} & 3\g^2\rho_x^2\left(1-\frac{\|x-y_x\|}{2\vrho_\g(x)}\right)-\g^2\rho_x^2\frac{\|x-y_x\|}{2\vrho_\g(x)}\\&=&\g^2\rho_x^2\left(3-\frac{2\|x-y_x\|}{\vrho_\g(x)}\right). \end{eqnarray*}
Add to this that $\|s-x\|\geq d_S(x)=\rho_x$, we deduce, using \eqref{last2}, that \begin{eqnarray*} \<s-x, s-x-2\vrho_\g(x)\zeta_x\> &\geq& \rho_x^2\left(1-2\g^2+(3\g^2-1)\frac{\vrho_\g(x)}{\|x-y_x\|}\right) \\ &\geq&  \rho_x^2 \left(\frac{1-\g^2}{2}\right)\;\;\;\;\left[\hbox{since}\,\;\g\geq \frac{1}{\sqrt{3}}\;\,\hbox{and}\,\;\frac{\vrho_\g(x)}{\|x-y_x\|}\geq \frac{1}{2}\right] \\&>&0. \end{eqnarray*}
The proof of Theorem \ref{mainth} is terminated.  \hfill $\qedsymbol$


\end{document}